\documentclass[11pt,twoside,reqno]{amsart}
\usepackage[foot]{amsaddr}

\usepackage{amssymb,amsfonts,amsmath,amsthm,bm}
\usepackage[mathscr]{eucal}
\usepackage{bbm}
\usepackage{stmaryrd}
\usepackage{fullpage}
\usepackage[margin=1in]{geometry}
\usepackage{blkarray}
\usepackage{array}
\usepackage{mathdots}
\usepackage{arydshln}
\usepackage{mathdots}

\usepackage{etoolbox}
\patchcmd{\section}{\scshape}{\bfseries}{}{}
\makeatletter
\renewcommand{\@secnumfont}{\bfseries}
\makeatother

\theoremstyle{definition}
\newtheorem{thm}{Theorem}
\newtheorem*{thm*}{Theorem}
\newtheorem{prop}[thm]{Proposition} 
\newtheorem{lem}[thm]{Lemma}
\newtheorem{defi}[thm]{Definition} 

\newtheorem{rem}[thm]{Remark}

\newtheorem{conj}[thm]{Conjecture}

\newtheorem{fact}[thm]{Fact}

\numberwithin{equation}{section}
\numberwithin{thm}{section}

\usepackage{enumerate}
\usepackage[shortlabels]{enumitem}

\usepackage{hyperref}
\usepackage[dvipsnames]{xcolor}
\newcommand\myshade{85}
\colorlet{mylinkcolor}{red}
\colorlet{mycitecolor}{blue}
\colorlet{myurlcolor}{Aquamarine}
\hypersetup{
	linkcolor  = mylinkcolor,
	citecolor  = mycitecolor,
	urlcolor   = myurlcolor!\myshade!black,
	colorlinks = true,
}

\usepackage{pythonhighlight}

\usepackage{tikz}

\usetikzlibrary{arrows,math,calc}
\usepackage[all,cmtip]{xy}

\newcommand{\la}[1]{\mathfrak{#1}}

\newcommand{\ZZ}{\mathbb{Z}}

\newcommand{\CC}{\mathbb{C}}

\newcommand{\sW}{\mathscr{W}}

\newcommand{\sC}{\mathscr{C}}

\newcommand{\fS}{\mathfrak{S}}

\DeclareMathOperator{\ch}{ch}

\DeclareMathOperator{\qdim}{qdim}
\DeclareMathOperator{\id}{Id}

\DeclareMathOperator{\sR}{\mathscr{R}}

\newcommand{\weylvec}{\delta}

\DeclareMathOperator{\typeA}{A}

\allowdisplaybreaks

\begin{document}

\date{}

\title{Coloured $\la{sl}_r$ invariants of torus knots and characters of $\sW_r$ algebras}
\author{Shashank Kanade}
\address{Department of Mathematics, University of Denver, Denver, CO 80208}
\email{shashank.kanade@du.edu}
\thanks{The author gratefully acknowledges the support of Simons Collaboration
 Grant for Mathematicians \#636937, and thanks T.\ Creutzig, K.\ Hikami,
 A.\ Milas and R.\ Osburn for valuable comments.
}

\begin{abstract} 
 Let $p<p'$ be a pair of coprime positive integers. In this
 note, generalizing Morton's work in the case of $\la{sl}_2$, we give a
 formula for the $\la{sl}_r$ Jones invariants of torus knots $T
 (p,p')$ coloured with the finite-dimensional irreducible representations
 $L_r(n\Lambda_1)$. When $r \leq p$, we show that appropriate limits
 of the shifted (non-normalized, framing dependent) invariants calculated
 along $L_r(nr\Lambda_1)$ are essentially the characters of certain minimal
 model principal $\sW$ algebras of type $\typeA$, namely, $\sW_r(p,p')$, up
 to some factors independent of $p$ and $p'$ but depending on $r$. In
 particular, these limits are essentially modular. We expect these limits to
 be the $0$-tails of corresponding sequences of invariants. At the end, we
 formulate a conjecture on limits for $p<r$.
\end{abstract}
\maketitle

\section{Introduction}

Let $\la{g}$ be a finite dimensional simple Lie algebra over $\CC$, let
$\lambda$ be a dominant integral weight, and let $L_{\la{g}}(\lambda)$ be the
finite dimensional irreducible representation with highest weight $\lambda$.

In \cite{RosJon-torus}, Rosso and Jones provided a formula for
calculating the $L_{\la{g}}(\lambda)$ coloured invariants of torus knots $T
(p,p')$ where $1\leq p<p'$ and $(p,p')=1$. In \cite{Mor-coloured}, Morton
used this formula and calculated coloured $\la{sl}_2$ Jones invariants for
all torus knots $T(p,p')$. 

We show that Morton's calculations generalize straight-forwardly to the case
when $\la{g}=\la{sl}_r$ and $\lambda=n\Lambda_1$ (or $\lambda=n\Lambda_
{r-1}$). The main ingredient driving these calculations is that all weights
of $L_r(n\Lambda_1)$ have multiplicity one.

In the case when $r\leq p$, we then show that as $n\rightarrow\infty$, the
resulting coloured (non-normalized, framing dependent) invariants have
periodic limits. This period is related to the group $P_r/Q_r$ where $P_r$
and $Q_r$ respectively denote the weight and root lattices of $\la{sl}_r$. In
particular, we show that if $n\rightarrow \infty$ along multiples of $r$ then
the limit is essentially the character of the vertex operator algebra $\sW_r
(p,p')$, up to some factors independent of $p,p'$ but depending on $r$. We
expect these limits to be the $0$-tails (see Definition \ref{defi:tail} below)
of the corresponding
sequence of invariants, but at the moment, we do not have a
full proof of this. The case $p<r$ is much more involved -- our formula for
(non-normalized, framing dependent) Jones polynomial involves a lot of
cancellation, and their limit is $0$, however, their tails are not. This case
seems to have a lot of nice structure, see Conjecture \ref{conj:p<r} below.
We hope to study this in detail in future.

The algebras $\sW_r(p,p')$ are known as principal $\sW$ algebras of type
$\typeA$ and due to results of Arakawa, \cite{Ara-c2} and \cite
{Ara-princrat}, they satisfy important properties, namely, $C_2$-cofiniteness
and rationality. Results of \cite{DonLinNg} now tell us that
these characters are modular invariant with respect to appropriate congruence
subgroups up to some factor $q^t$ depending on the central charge of $\sW_r
(p,p')$. 

By now, there is an extensive literature on heads, tails and the stability of
coloured Jones polynomials. We review some of the major highlights relevant
for the present paper. The $0$-stability of coefficients of coloured $\la
{sl}_2$ Jones polynomials seems to have been first conjectured by Dasbach and
Lin in \cite{DasLin}, 
and proved for alternating and for adequate links by Armond \cite{Arm-headtail}
using skein-theoretic methods.
A general notion of $k$-stability was
introduced in \cite{GarLe-Nahm} where Garoufalidis and L\^e proved that
coloured $\la{sl}_2$ Jones polynomials of alternating links are $k$-stable
for all $k$ (see also \cite{Hal-stability} and \cite{Bei-pretzel} for
higher order stability).
For some related, but purely $q$-series-theoretic studies, see 
\cite{And-knots}, \cite{KeiOsb}, \cite{BeiOsb}, etc. In \cite{GarVuo} 
Garoufalidis and Vuong conjectured that for a knot $K$, simple Lie
algebra $\la{g}$ and dominant integral weight $\lambda$, the sequence of
coloured $\la{g}$-invariants $J_K(L_\la{g}(n\lambda))$ has a property called
cyclotomic stability. Roughly speaking, this property captures the periodic
nature of limits as mentioned above. In the same paper \cite{GarVuo}, authors
proved this conjecture for all torus knots and and all rank $2$ Lie algebras.
Recently, Yuasa \cite{Yua-stability} has proved $0$-stability for $L
(n\Lambda_1)$-coloured $\la{sl}_3$ Jones polynomials for minus-adequate links
using skein theory of $A_2$; see also \cite{Yua-q} and \cite{Yua-torus}.

We note that several other relations between invariants of knots, manifolds
and characters of various algebras are known in the literature. For instance,
in \cite{BriMil-singlet} tails for torus links $(2,2p)$ were related to
characters of (irrational) singlet vertex operator algebras. Much earlier,
in \cite{HikKir}, Kashaev invariants of torus knots $T(p,p')$ were shown to be
related to the Eichler integrals of $\mathrm{Vir}(p,p')=\sW_2(p,p')$ minimal
model characters. There are even more relations between certain quantum
invariants of 3-manifolds and characters of certain logarithmic VOAs found recently,
\cite{CheEtAl}. Finally, we mention \cite{EtiGorLos} where HOMFLYPT
polynomials of torus knots were related to characters of certain rational
Cherednik algebras.

In another direction, we mention that recently, there has been much interest
(see for example \cite{AndSchWar}, \cite{CorDouUnc}, \cite{FodWel}, \cite
{War}, \cite{KanRus-cylindric} and \cite{Tsu}) in understanding the
combinatorics of $\sW_3(3,p)$ characters. In general, $\sW_r(p,p')$ 
characters are related to cylindric partitions; see \cite{GesKra}, 
\cite{FodWel}, etc. It will be very interesting to see if knot
theory could provide other combinatorial approaches to these characters
similar to how walks along the knot $T(2,2n+1)$ give rise to fermionic
characters of $\sW_2(2,2n+1)$ \cite{ArmDas-RR}. This has been our primary
motivation for the present study.

This paper is organized as follows. We shall review some basic facts about Lie
algebras $\la{sl}_r$ in Section \ref{sec:lie}. We will then review characters
of $\sW_r(p,p')$ algebras in Section \ref{sec:w} where we study the cases
$r=p$ and $p<r$ in detail. Then, in Section \ref{sec:RJ} we explain the
Rosso--Jones formula which governs coloured $\la{sl}_r$ invariants of torus
knots. Section \ref{sec:poly} presents the main calculation of $J_{T(p,p')}
(L_r(n\Lambda_1))$, generalizing Morton's arguments from \cite
{Mor-coloured}. Lastly, we study the $n\rightarrow \infty$ limits in
Section \ref{sec:limits}. As mentioned above, the limits are trivial for
$p<r$, but the tails are not. We present the example of $r=3, p=2$ in detail
and show that our results match those given in \cite{GarVuo}. Finally,
we formulate a general conjecture in the case $p<r$.

Finally, we now explain certain choices made in this paper. 
Firstly, some of our main results are presented with 
respect to non-normalized and framing dependent invariants $J_{T(p,p')}$ (where 
$T(p,p')$ is assumed to have writhe $pp'$.) This has the nice consequence that
for $p<p'$ with $(p,p')=1$, 
$J_{T(p,p')}(L_r(kr\Lambda_1))\in \ZZ[[q]]$ for all $k$s, 
see Remark \ref{rem:jonesintegral}. 
Secondly, it is common in vertex operator algebra literature to denote the 
parameters of $\sW_r$ algebras as $p,p'$. While this becomes cumbersome especially 
in complicated formulas, we keep this notation since the role of parameters
$p,p'$ is largely interchangeable and thus they merit similar notation.
Indeed, we have $\sW_r(p,p')\cong \sW_r(p',p)$ and we also have 
$T(p,p')=T(p',p)$. 
Thirdly, all of the analysis presented here also works 
with weights $n\Lambda_{r-1}$, however, for convenience, we adhere to $n\Lambda_1$.

\section{Lie algebras \texorpdfstring{$\la{sl}_r$}{sl\_r}}
\label{sec:lie}

Throughout, we will be working with the Lie algebra $\la{sl}_r$  of trace $0$ $r\times r$ complex matrices with $r\geq 2$. 
We will let the Cartan subalgebra $\mathfrak{h}$ to be the set of diagonal matrices. Let $(\cdot,\cdot)$ be the trace form
on $\la{sl}_r$ which is symmetric, invariant and non-degenerate and stays non-degenerate on $\la{h}$. 
We will identify $\mathfrak{h}$ and $\mathfrak{h}^\ast$ via this form. 
In the usual way, we will embed $\mathfrak{h}^\ast$
into the vector space spanned by symbols $\epsilon_1,\dots,\epsilon_r$ such that $(\epsilon_i,\epsilon_j)=\delta_{ij}$.
For $1\leq i\leq r-1$ we will have the simple roots  $\alpha_i=\epsilon_i-\epsilon_{i+1}\in\la{h}^*$  and the fundamental weights
$\Lambda_i\in\la{h}^*$  such that $(\Lambda_i,\alpha_j)=\delta_{ij}$.
The sets of roots and fundamental weights both form bases of $\la{h}^*$.
We have the following sets of positive, negative and all roots:
\begin{align*}
\Phi_r^{\pm}=\{  \pm(\alpha_i+\cdots+\alpha_j) \,\vert\, 1\leq i\leq j\leq r-1\},\quad \Phi_r=\Phi_r^+\cup \Phi_r^-.
\end{align*}
The Weyl vector will be denoted by $\weylvec\in \la{h}^*$ and it is uniquely determined by:
\begin{align*}
(\delta,\alpha_i)=1\quad\mathrm{for}\,\,\mathrm{all}\,\,\,1\leq i\leq r-1.
\end{align*}
We have, with $1\leq i\leq r-1$:
\begin{align*}
\weylvec&=\frac{(r-1)}{2}\epsilon_1 + \frac{(r-3)}{2}\epsilon_2  +\cdots + \frac{(1-r)}{2}\epsilon_r,\\
\Lambda_i &= \frac{r-i}{r}(\epsilon_1+\cdots+\epsilon_i) - \frac{i}{r}(\epsilon_{i+1}+\cdots+\epsilon_r).
\end{align*}

We will let $Q_r\subseteq P_r$ be the root and the weight lattices,
respectively. Recall that $Q_r$ is an even lattice, i.e., $\|\lambda\|^2\in
2\ZZ$ for all $\lambda\in Q_r$. $P_r^+$ will denote the set of dominant
integral weights, i.e., $\lambda\in P_r^+$ iff $(\lambda,\alpha_i)\in \ZZ_
{\geq 0}$ for all $i$. We will let $P^\circ_r$ denote the dominant integral
weights that belong to the interior of the fundamental Weyl chamber, i.e.,
$\lambda\in P_r^\circ$ iff $(\lambda,\alpha_i)\in \ZZ_{> 0}$ for all $i$.

The Weyl group of $\la{sl}_r$ acts on $\mathfrak{h}^\ast$ and is isomorphic to the symmetric group $\fS_r$.
$\fS_r$ acts on $\epsilon_i$s by permuting the subscripts and this determines its action on $\mathfrak{h}^\ast$.
Recall that $Q_r$, $P_r$ and the form $(\cdot,\cdot)$ are all $\fS_r$ invariant.
With respect to the simple transpositions $(i,i+1)$, the length of the shortest expression for $w$ will be denoted
as $\ell(w)$. The sign representation of $\fS_r$ equals $(-1)^{\ell(w)}$.

For $\lambda\in P^+_r$, let $L_r(\lambda)$ denote the irreducible $\la{sl}_r$ module with highest weight $\lambda$.
The modules $L_r(n\Lambda_1)$ for $n\geq 0$ will be important to us in this note. We denote the set of weights of $L_r(n\Lambda_1)$ by $\Pi_{r,n}$.
The following fact about $\Pi_{r,n}$ will be crucial for us, and so we provide a quick proof; see also \cite{FulHar}.

\begin{fact}
We have
$\lambda \in \Pi_{r,n}$ iff
$\lambda=n\Lambda_1-a_1\alpha_1-\cdots-a_{r-1}\alpha_{r-1}$ with 
$n\geq a_1\geq \cdots a_{r-1}\geq 0.$
Additionally, every weight of $L_r(n\Lambda_1)$ occurs with multiplicity $1$. 
\end{fact}
\begin{proof}
Successively using the unbrokenness of the $\la{sl}_2$ strings with respect to the roots $\alpha_1$, $\alpha_2$, \dots, $\alpha_{r-1}$,
we see that each $\lambda=n\Lambda_1-a_1\alpha_1-\cdots-a_{r-1}\alpha_{r-1}$ with $n\geq a_1\geq \cdots a_{r-1}\geq 0$ 
belongs to $\Pi_{r,n}$. This implies that the cardinality of $\Pi_{r,n}$ is at least the $n$th $(r-1)$-dimensional triangular number, 
which equals $\binom{n+r-1}{r-1}$. This is in fact the same as the dimension of $L_r(n\Lambda_1)$ as can be seen quickly from the Weyl dimension
formula \cite[Eq.\ 15.17]{FulHar}. Consequently, $\Pi_{r,n}$ is precisely the set of such weights and moreover, each of them appears with multiplicity exactly $1$.
\end{proof}

In fact, similar argument works with $L_r(n\Lambda_{r-1})$. The weights of
this module are of the form $n\Lambda_{r-1}-a_{r-1}\alpha_{r-1}-\cdots-a_{1}\alpha_{1}$ 
for $n\geq a_{r-1}\geq \cdots\geq a_1\geq 0$ with each weight
having multiplicity exactly $1$. We will stick with $n\Lambda_1$ throughout
this note but all the analysis also works for $n\Lambda_{r-1}$.

Noting that $r\Lambda_1=(r-1)\alpha_1+(r-2)\alpha_2+\cdots+\alpha_{r-1}\in Q_r$,  
following two further facts are clear.

\begin{fact}
We have that $\Pi_{r,n}\subset\Pi_{r,n+r}$.
\end{fact}
\begin{fact}
We have that $\bigcup_{j\geq 0}\Pi_{r,n+jr} = n\Lambda_1+Q_r= t\Lambda_1+Q_r$
where $t$ is the residue of $n$ modulo $r$.
\end{fact}

\section{Characters of \texorpdfstring{$\sW_r(p,p')$}{W\_r(p,p')} algebras}
\label{sec:w} 

In this section, we give character formulas for irreducible modules of the
vertex operator algebra $\sW_r(p,p')$ following \cite{AndSchWar} and \cite
{FodWel}.  We then study the cases $r=p$ and $p<r$ separately, as they have
interesting ramifications when we consider $\la{sl}_r$ invariants of torus
knots later. Note that this character formula requires $r\leq p,p'$ with
$p,p'$ coprime, but we may continue to substitute other values of the
parameters in the characters and treat them purely as $q$-series.
 
\subsection{The characters}
Let $2\leq r\leq p < p'$ be integers with $p,p'$ relatively prime.

Given a non-negative integer $k$, we will let $P^+_{r,k}$ be the of dominant integral weights $\lambda$ such that
$(\lambda,\theta)\leq k$ with $\theta$ being the highest root $\alpha_1+\cdots+\alpha_{r-1}$. This set is in bijection with the highest weights of integrable, irreducible, highest-weight,
level $k$ modules for the affine Lie algebra $\widehat{\la{sl}_r}$. 

The principal $\sW$ algebra of $\la{sl}_r$ with parameters $(p,p')$, denoted as $\sW_r(p,p')$ 
has inequivalent simple representations parametrized by $(\xi,\zeta)\in {P}^+_{r,p-r}\times {P}^+_{r,p'-r}$.
With $\eta(q)=q^{1/24}(q;q)_\infty$ denoting the Dedekind eta function,
the corresponding characters are given as follows, 
\cite[Eq.\ (91)]{FodWel}:
\begin{align}
\chi^{r,p,p'}_{\xi,\zeta}
&=\dfrac{1}{\eta(q)^{r-1}}
\sum_{\alpha\in Q_r}\sum_{\sigma\in\fS_r}
(-1)^{\ell(\sigma)}q^{\frac{1}{2}pp'\| \alpha- (\xi+\weylvec)/p + \sigma(\zeta+\weylvec)/p'\|^2} 
\nonumber\\
&=
\dfrac{q^{\frac{1}{2}pp'\|\frac{\xi+\weylvec}{p} - \frac{\zeta+\weylvec}{p'}\|^2  -\frac{r-1}{24}}}{(q;q)_\infty^{r-1}}
\sum_{\alpha\in Q_r}\sum_{\sigma\in\fS_r}
(-1)^{\ell(\sigma)}q^{
\frac{1}{2}pp' 
\| \alpha\|^2 
-p'(\alpha,\xi+\weylvec)
+p(\alpha,\sigma(\zeta+\weylvec))
-(\xi+\weylvec,\sigma(\zeta+\weylvec)-(\zeta+\weylvec))
}\label{eqn:Wcharnumer} \\
&=
{q^{\frac{1}{2}pp'\|\frac{\xi+\weylvec}{p} - \frac{\zeta+\weylvec}{p'}\|^2  -\frac{r-1}{24}}}
(q;q)_\infty \sC^{r}_{\xi,\zeta}.
\nonumber
\end{align}
where $\sC^{r}_{\xi,\zeta}\in 1+q\ZZ[[q]]$ is the generating function
of certain kinds of cylindric partitions, \cite[Eqn.\ (27)]{FodWel}.
We may now normalize the character so that it belongs to $1+q\ZZ[[q]]$:
\begin{align}
\overline{\chi}^{r,p,p'}_{\xi,\zeta}=
(q;q)_\infty \sC^{r}_{\xi,\zeta}\in 1+q\ZZ[[q]].
\label{eqn:chi1plus}
\end{align}
The choice $\xi=\zeta=0$ corresponds to the vertex operator algebra $\sW_r(p,p')$, and we record the following
normalized character:
\begin{align}
\overline{\chi}(\sW_r(p,p'))&=\overline{\chi}^{r,p,p'}_{0,0}
=
\dfrac{1}{(q;q)_\infty^{r-1}}
\sum_{\alpha\in Q_r}\sum_{\sigma\in\fS_r}
(-1)^{\ell(\sigma)}q^{
\frac{1}{2}pp' 
\| \alpha\|^2 
-p'(\alpha,\weylvec)
+p(\alpha,\sigma(\weylvec))
-(\weylvec,\sigma(\weylvec)-\weylvec)
}.
\label{eqn:Wnormchar}
\end{align}

Given a $\mu\in P_r$, it will be beneficial for us to define the $\mu$-shifted character:
\begin{align*}
\overline{\chi}^{r,p,p';\,\mu}_{\xi,\zeta}=
\dfrac{1}{(q;q)_\infty^{r-1}}
\sum_{\alpha\in Q_r+\mu}\sum_{\sigma\in\fS_r}
(-1)^{\ell(\sigma)}q^{
\frac{1}{2}pp' 
\| \alpha\|^2 
-p'(\alpha,\xi+\weylvec)
+p(\alpha,\sigma(\zeta+\weylvec))
-(\xi+\weylvec,\sigma(\zeta+\weylvec)-(\zeta+\weylvec))
}
\end{align*}
These shifted characters are unchanged if we vary $\mu$ in a coset of $Q_r$. For
$\mu\in Q_r$, we have:
\begin{align*}
\overline{\chi}^{r,p,p';\,\mu}_{\xi,\zeta}=\overline{\chi}^{r,p,p'}_{\xi,\zeta}.
\end{align*}

\subsection{The case \texorpdfstring{$r=p$}{r=p}}

We now show that the case $r=p$ is somewhat special in that all the $\mu$-shifted  characters coincide
up to a sign.
In particular, when $\mu=k\Lambda_1$, it will help us in showing that shifted Jones invariants of torus knots $T(r,p')$ with $r<p'$ and $(r,p')=1$
have a well-defined (non-periodic) limit with respect to the representations $L_r(n\Lambda_1)$ of $\la{sl}_r$.

To this end, define the following for $p'>r$ with $(r,p')=1$, $\mu\in P_r$,  $\zeta\in P_{r,p'-r}^+$ and $\sigma\in \fS_r$:
\begin{align}
\Gamma^{r,r,p';\,\mu}_{0,\zeta}&(\sigma)=
\sum_{\alpha\in Q_r+\mu}
(-1)^{\ell(\sigma)}q^{
\frac{1}{2}rp'
\|\alpha\|^2 - 
p'(\alpha,\weylvec)
+r(\alpha,\sigma(\zeta+\weylvec))
- (\weylvec,\sigma(\zeta+\weylvec)-\zeta-\weylvec)
}\nonumber\\
&=
q^{\frac{1}{2}rp'
\|\mu\|^2 
+(\mu,-p'\weylvec+r\sigma(\zeta+\weylvec))
- (\weylvec,\sigma(\zeta+\weylvec)-\zeta-\weylvec)
}
\sum_{\alpha\in Q_r}
(-1)^{\ell(\sigma)}q^{
\frac{1}{2}rp'
\|\alpha\|^2 
+
(\alpha,rp'\mu-p'\weylvec+r\sigma(\zeta+\weylvec))
}
\label{eqn:Gammadef}
\end{align}
Before we proceed to manipulate these sums, we record the following useful property.

\begin{lem}
\label{lem:Prw}
Given any $\mu\in P_r$, 
there exists a $w\in \fS_r$ and $\lambda\in Q_r$ such that:
\begin{align}
r\mu - \delta+w\delta=r\lambda\in rQ_r.
\label{eqn:rP}
\end{align}
\end{lem}
\begin{proof}
If \eqref{eqn:rP} holds for $\mu\in P_r$ with $w\in \fS_r$ and $\mu'\in P_r$ with $w'\in \fS_r$,
then it holds for $\mu+w\mu'$ with $ww'\in \fS_r$ because:
\begin{align}
r(\mu+w\mu')-\weylvec+ww'\weylvec=r\mu-\weylvec+w\weylvec + rw\mu'-w\weylvec+ww'\weylvec
\in rQ_r.
\label{eqn:rPadd}
\end{align}
Let us now prove \eqref{eqn:rP} for $\mu=\Lambda_1$.
We have:
\begin{align*}
r\Lambda_1-\weylvec &= \frac{r-1}{2}\epsilon_1 +  \left(\frac{1-r}{2}\epsilon_2 + \frac{3-r}{2}\epsilon_3 + \cdots \frac{r-3}{2}\epsilon_r \right).
\end{align*}
We may thus take $w$ to be the cycle $\sigma=(r,1,2,3,\dots,r-1)$, so that $r\Lambda_1-\weylvec+\sigma\weylvec=0\in rQ_r$.

Iterating \eqref{eqn:rPadd} with $\mu=\mu'=\Lambda_1$ and $w=w'=\sigma$,
we see that \eqref{eqn:rP} holds for 
$\mu_k=\Lambda_1+w\Lambda_1+\cdots +w^{k-1}\Lambda_1$ with $\sigma^k$.
It is easy to see that for $1\leq k\leq r$,
\begin{align*}
\mu_k=
\Lambda_1+\sigma\Lambda_1+\cdots +\sigma^{k-1}\Lambda_1
=
\frac{(r-k)}{r}(\epsilon_1+\cdots+\epsilon_k)-\dfrac{k}{r}(\epsilon_{k+1}+\cdots+\epsilon_r)
\end{align*}
which equals $\Lambda_k$ if $1\leq k\leq r-1$ and is $0$ if $k=r$.

Any $\nu\in P_r$ can now be written as $\nu=\mu_k+\phi$ for some $1\leq k\leq r$ and $\phi\in Q_r$. 
Required statement for $\nu$ now follows by taking $\mu=\mu_k$, $\mu'=\phi$, $w=\sigma^k$ and $w'=\id$ in \eqref{eqn:rPadd}.
\end{proof}

We now have the following result.
\begin{prop}
Fix $\mu\in P_r$,  and let $\lambda\in Q_r$  and $w\in \fS_r$ be as in Lemma \ref{lem:Prw}. Then, for all $\sigma\in\fS_r$
and $\zeta\in P^+_{r,p'-r}$ we have:
\begin{align*}
(-1)^{\ell({w})}
\cdot
\Gamma^{r,r,p';\,\mu}_{0,\zeta}(\sigma)
=
\Gamma^{r,r,p';\,0}_{0,\zeta}(w^{-1}\sigma).
\end{align*}
Consequently, we have:
\begin{align}
(-1)^{\ell({w})}&
\cdot
\overline{\chi}^{r,r,p';\,\mu}_{0,\zeta}
=\overline{\chi}^{r,p,p'}_{0,\zeta}.
\label{eqn:r=p}
\end{align}
Note that the sign $(-1)^{\ell(w)}$ depends solely on $\mu$.
\end{prop}
\begin{proof}
Pick and fix $\sigma\in \fS_r$.
For convenience, denote $v=w^{-1}\sigma$.
We then have:
\begin{align*}
\Gamma&{}^{r,r,p';\,0}_{0,\zeta}(v)\\
&=
\sum_{\alpha\in Q_r}
(-1)^{\ell(v)}q^{
\frac{1}{2}rp'
\|\alpha\|^2 
- 
p'(\alpha,\weylvec)
+r(\alpha,v(\zeta+\weylvec))
- (\weylvec,v(\zeta+\weylvec)-\zeta-\weylvec)
}\\
&=\sum_{\alpha\in w^{-1}Q_r+\lambda}
(-1)^{\ell(v)}q^{
\frac{1}{2}rp'
\|\alpha\|^2 
- 
p'(\alpha,\weylvec)
+r(\alpha,v(\zeta+\weylvec))
- (\weylvec,v(\zeta+\weylvec)-\zeta-\weylvec)
}
\\
&=
q^{
- (\weylvec,v(\zeta+\weylvec)-\zeta-\weylvec)
+\frac{1}{2}rp'
\|\lambda\|^2 
+(\lambda,-p'w\weylvec+r\sigma(\zeta+\weylvec))
}
\sum_{\alpha\in Q_r}
(-1)^{\ell(v)}q^{
\frac{1}{2}rp'
\|\alpha\|^2 
+
(\alpha,rp'\lambda-p'w\weylvec+r\sigma(\zeta+\weylvec))
}\\
&\overset{\eqref{eqn:rP}}{=}
(-1)^{\ell({w})}
q^{
- (\weylvec,v(\zeta+\weylvec)-\zeta-\weylvec)
+\frac{1}{2}rp'
\|\lambda\|^2 
+(\lambda,-p'w\weylvec+r\sigma(\zeta+\weylvec))
}
\sum_{\alpha\in Q_r}
(-1)^{\ell(\sigma)}q^{
\frac{1}{2}rp'
\|\alpha\|^2 
+
(\alpha,rp'\mu-p'\weylvec+r\sigma(\zeta+\weylvec))
}
\\
&=
(-1)^{\ell({w})}
q^{
 \frac{1}{2}rp'
\left( \|\lambda\|^2 - \|\mu\|^2 \right)
-(\lambda,p'w\weylvec)
+(\mu,p'\weylvec)
-(\weylvec,v(\zeta+\weylvec))
+(\lambda,r\sigma(\zeta+\weylvec))
+(\weylvec,\sigma(\zeta+\weylvec))
-(\mu,r\sigma(\zeta+\weylvec))
}
\cdot
\Gamma^{r,r,p';\mu}_{0,\zeta}(\sigma).
\end{align*}
Now it is easy to check using $v=w^{-1}\sigma$ and $r\weylvec-\weylvec+w\weylvec=r\lambda$ that:
\begin{align*}
-(\weylvec,v(\zeta+\weylvec))
+(\lambda,r\sigma(\zeta+\weylvec))
+(\weylvec,\sigma(\zeta+\weylvec))
-(\mu,r\sigma(\zeta+\weylvec))
=0,\\
\frac{1}{2}r
\left( \|\lambda\|^2 - \|\mu\|^2 \right)
-(\lambda,w\weylvec)
+(\mu,\weylvec)
=0.
\end{align*}
This immediately implies the first relation.
The second relation follows by summing the first relation 
over $\sigma\in\fS_r$ and finally dividing by $(q;q)_\infty^{r-1}$.
\end{proof}

\subsection{The case \texorpdfstring{$p<r$}{p<r}}
The character formula is not valid unless $r\leq p,p'$ and $(p,p')=1$.
We now show that if $p<r$ then these characters evaluate to $0$ as $q$-series.
This is related to the fact that our formulas for $\la{sl}_r$ invariants of corresponding torus knots involve a lot of cancellations
and thus one needs to be careful in calculating their tails.

We have the following lemma which essentially drives all the relevant cancellations.
\begin{lem}
\label{lem:cancellations}
Let $1\leq p$ and let $w\in \fS_r$. We have the following cases.
\begin{enumerate}
	\item If $p< r$ and $u=w(1,p+1)$ we have $u\weylvec-w\weylvec\in pQ_r$ and $(-1)^{\ell(u)}\neq (-1)^{\ell{(w)}}$.
	\item If $p=r-1$ then $u\weylvec-w\weylvec\in (r-1)Q_r$ iff $u=w$ or $u=w(1,r)$. In the latter case we have
	$u\weylvec-w\weylvec=-(r-1)w\theta$ where $\theta=\alpha_1+\cdots+\alpha_{r-1}\in\Phi_r$ is the highest root.
	\item If $p\geq r$ then $u\weylvec-w\weylvec\in pQ_r$ iff $u=w$.
\end{enumerate}
\end{lem}
\begin{proof}
For $0<p\leq r-1$, we have
\begin{align*}
(1,p+1)\weylvec-\weylvec=p(\epsilon_1-\epsilon_j)=p(\alpha_1+\cdots+\alpha_{j-1})\in pQ_r.
\end{align*}
This immediately proves the first part.
The only pair of \emph{different} coordinates of $\weylvec$ whose difference is an integer divisible by $r-1$ is $(r-1)\epsilon_1/2$ and $(1-r)\epsilon_r/2$.
Thus, if $p=r-1$, $\sigma\weylvec-\weylvec\in (r-1)Q_r$ iff $\sigma=1$ or $\sigma=(1,r)$. Consequently,
$u\weylvec-w\weylvec\in (r-1)Q_r$ iff $u=w$ or $u=w(1,r)$. 
The second part now follows easily.
For $p\geq r$, there is simply no way to subtract two different coordinates of $\weylvec$ to get an integer divisible by $p$.
Indeed, the biggest possible difference is $r-1$ when we swap $(r-1)/2$ with $(1-r)/2$. This implies the third part.
\end{proof}

We now have the following result.
\begin{prop}
\label{prop:p<r}
Let $1\leq p\leq r-1$, $\mu\in P_r$.
Then,
\begin{align*}
\overline{\chi}^{r,p,p';\,\mu}_{0,0}=0.
\end{align*}
\end{prop}
\begin{proof}
We will show below that for any fixed $w\in\fS_r$,
\begin{align*}
\Gamma^{r,p,p';\,\mu}_{0,0}(w)
=-\Gamma^{r,p,p';\,(1,p+1)\mu}_{0,0}((1,p+1)\cdot w).
\end{align*}
This will imply, upon summing over all $w$ that:
\begin{align*}
\overline{\chi}^{r,p,p';\,(1,p+1)\mu}_{0,0}
=-\overline{\chi}^{r,p,p';\,\mu}_{0,0}.
\end{align*}
But, noting that for $\mu\in P_r$ we have $(1,p+1)\mu-\mu\in Q_r$,
we also have:
\begin{align*}
\overline{\chi}^{r,p,p';\,(1,p+1)\mu}_{0,0}
=\overline{\chi}^{r,p,p';\,\mu}_{0,0},
\end{align*}
which now gives the required result.

Fix $w\in \fS_r$. We write $v=w^{-1}(1,p+1)$ and note that $(-1)^{\ell(v)}=-(-1)^{\ell(w)}$.
Further, from Lemma \ref{lem:cancellations}, we have $(1,p+1)\weylvec-\weylvec\in pQ_r$.
We may thus write 
\begin{align*}
(1,p+1)\weylvec-\weylvec=p\lambda
\end{align*}
for some $\lambda\in Q_r$.
We have:
\begin{align*}
\Gamma^{r,p,p';\,\mu}_{0,0}(w)
&=
\sum_{\alpha\in Q_r+\mu}
(-1)^{\ell(w)}q^{
\frac{1}{2}pp'
\|\alpha\|^2 
- p'(\alpha,\weylvec)
+p(\alpha,w\weylvec)
- (\weylvec,w\weylvec-\weylvec)
}
\\
&=
\sum_{\alpha\in Q_r+\mu}
(-1)^{\ell(w)}q^{
\frac{p'}{2p}
\| pw^{-1}\alpha-w^{-1}\weylvec\|^2 
+(pw^{-1}\alpha-w^{-1}\weylvec,\weylvec)
+\|\weylvec\|^2
-\frac{p'}{2p}\|\weylvec\|^2
}\\
&=
-\sum_{\alpha\in Q_r+\mu}
(-1)^{\ell(v)}q^{
\frac{p'}{2p}
\| pw^{-1}(\alpha+\lambda)-v\weylvec\|^2 
+(pw^{-1}(\alpha+\lambda)-v\weylvec,\weylvec)
+\|\weylvec\|^2
-\frac{p'}{2p}\|\weylvec\|^2
}\\
&=
-\sum_{\alpha\in Q_r+w^{-1}\mu}
(-1)^{\ell(v)}q^{
\frac{p'}{2p}
\| p\alpha-v\weylvec\|^2 
+(p\alpha-v\weylvec,\weylvec)
+\|\weylvec\|^2
-\frac{p'}{2p}\|\weylvec\|^2
}\\
&=
-\sum_{\alpha\in Q_r+v^{-1}w^{-1}\mu}
(-1)^{\ell(v)}q^{
\frac{p'}{2p}
\| pv\alpha-v\weylvec\|^2 
+(pv\alpha-v\weylvec,\weylvec)
+\|\weylvec\|^2
-\frac{p'}{2p}\|\weylvec\|^2
}\\
&=-\Gamma^{r,p,p';\,(1,p+1)\mu}_{0,0}((1,p+1)w).
\end{align*}
\end{proof}

\section{Torus knots and the Rosso--Jones formula}
\label{sec:RJ}

Fix the Lie algebra $\la{sl}_r$.  

Let $\sR$ denote the Grothendieck ring over $\CC$ of finite dimensional representations of $\la{sl}_r$.
As a vector space over $\CC$ it has a basis of finite dimensional irreducible modules, but as an algebra over $\CC$, it is generated by those modules whose highest weights are the fundamental weights.
This ring is isomorphic to the ring $\ch(\sR)=\left(\CC[x_1,\cdots,x_r]/(x_1\cdots x_r-1)\right)^{\fS_r}$ where $\fS_r$ acts by permuting the variables. Naturally, the isomorphism is provided by considering characters of the modules.
Let us also consider the bigger ring $\widehat{\ch}(\sR)=\CC[x_1,\cdots,x_r]/(x_1\cdots x_r-1)$ which is isomorphic to the group algebra of $P_r$.
This isomorphism is as follows. Given a $\lambda=a_1\Lambda_1+\cdots+a_r\Lambda_{r-1}\in P_r$, 
the corresponding monomial in $\widehat{\ch}(\sR)$ is:
\begin{align}
\mathbf{x}^\lambda = \prod_{i=1}^{r-1}x_i^{a_i+a_{i+1}+\cdots+a_{r-1}}.
\end{align}
The characters of (finite-dimensional) irreducible modules are given by Weyl's character formula:
\begin{align*}
\ch(L_r(\lambda))=\dfrac{\sum_{w\in\fS_r}(-1)^{\ell(w)} \mathbf{x}^{w(\lambda+\weylvec)}}{\sum_{w\in\fS_r}(-1)^{\ell(w)} \mathbf{x}^{w\weylvec}}
=\dfrac{N(\lambda)}{\Delta_r}
\end{align*}
The numerator and denominator of this formula do not belong to $\ch(\sR)$ as they are alternating functions under the action of $\fS_r$.
They do, however, naturally live in the the bigger ring $\widehat{\ch}(\sR)$
since it is isomorphic to the group algebra of $P_r$.

For a dominant integral weight $\lambda$
and a knot $K$ along with a framing $f$,
let $J^f_K(L_r(\lambda))$ denote the (framing dependent) Reshetikhin-Turaev invariant
of $K$ with strands corresponding to $L_r(\lambda)$. Let $J^u_K(L_r(\lambda))$
denote the (framing independent) coloured Jones invariant, i.e., the framing of $K$ 
is altered to zero by introducing appropriate number of twists/curls.
Further let $J^{f,\mathbbm{1}}_K(L_r(\lambda))$ (resp.\ $J^{u,\mathbbm{1}}_K(L_r(\lambda))$) denote the 
framing dependent (resp.\ framing independent)
normalized invariant that satisfies
$J^{f,\mathbbm{1}}_\circ( L_r(\lambda) )=1$ (resp.\ $J^{u,\mathbbm{1}}_\circ( L_r(\lambda) )=1$) where $\circ$ is the unknot.

For a knot $K$, we will view $J^f_K, J^u_K, J^{f,\mathbbm{1}}_K, J^{u,\mathbbm{1}}_K$ as functions 
\begin{align*}
J^\bullet_K : \mathscr{R}&\rightarrow \ZZ[q^t,q^{-t}]\\
L_r(\lambda)&\mapsto J^\bullet_K(L_r(\lambda)).
\end{align*}
where $t$ is certain appropriate fraction.
We will now use this upgraded notation where the argument of $J^\bullet_K$ is any element of $\sR$.
By an abuse of the notation, we also use $J^\bullet_K$ to denote corresponding maps on $\ch(\sR)$.

Let $p<p'$ be a pair of coprime positive integers.

Given a knot $K$ with framing $f$, $K{(p,p')}$ denotes the $(p,p')$ cabling of $K$, with framing $\overline{f}$ inherited from 
the framing $f$ of $K$.
Rosso--Jones formula relates coloured $\la{sl}_r$ invariants of $K(p,p')$ with that of $K$.
Two functions on $\sR$ enter into the Rosso--Jones formula \cite{RosJon-torus}, as explained by \cite{Mor-coloured}.
The first is the \emph{twist} which is the action of ribbon element on 
the modules. 
On (finite-dimensional) irreducible modules, it acts by a scalar, given by ($\lambda\in P_r^+$):
\begin{align*}
\theta_{L_r(\lambda)}= q^{\frac{1}{2}(\lambda, \lambda+2\weylvec)}.
\end{align*}
Second map is the \emph{$p$th Adams operation}. 
On $\ch(\sR)$ it acts by:
\begin{align*}
\psi_p(x_i) = x_i^p
\end{align*}
Since $\sR$ and $\ch(\sR)$ are isomorphic, we will use the same notation $\psi_p$ to denote the operation on $\sR$.
In fact, $\psi_p$ is also well-defined on the bigger ring $\widehat{\ch}(\sR)$.

Now, the Rosso--Jones \cite{RosJon-torus} formula, as explained in \cite{Mor-coloured}, states:
\begin{thm}
Let $p,p'$ be a pair of positive coprime integers. As maps on $\sR$ we have:
\begin{align}
J_{K(p,p')}^{\overline{f}} = J_K^f\circ \theta^{p'/p}\circ \psi_p.
\end{align}
\end{thm}
Let us specialize to the case $K=\circ$, the unknot with zero framing $u$.
In this case,  $\circ(p,p')$ gives rise to a framed torus knot $T(p,p')$ with writhe $pp'$. Call this framing $\overline{u}$.

Now, $J_\circ^u$ acts on the ring $\sR$ simply as the \emph{quantum dimension}.
That is,
\begin{align}
J_\circ^u(L_r(\lambda))=\qdim(L_r(\lambda))=
\prod_{\alpha\in\Phi_r^+}\dfrac{q^{\frac{1}{2}(\lambda+\weylvec,\alpha)}-q^{-\frac{1}{2}(\lambda+\weylvec,\alpha)}}{q^{\frac{1}{2}(\weylvec,\alpha)}-q^{-\frac{1}{2}(\weylvec,\alpha)}}.
\end{align}
Actually, this is known as the principal specialization of the character, and on $\widehat{\ch}(\sR)$ it can be alternatively defined by:
\begin{align}
\qdim(\mathbf{x}^\lambda)= q^{(\lambda,\delta)}.
\label{eqn:qdimx}
\end{align}

Let $\lambda$ be a dominant integral weight, and suppose that
\begin{align}
\psi_p(L_r(\lambda)) = \sum_{\mu\in P^+_r} m^{\mu}_{\lambda,p} L_r(\mu).
\end{align}
Only finitely many $m^{\mu}_{\lambda,p}\in\CC$ are non-zero and they are in fact all integers.
Then, we have:
\begin{align}
J_{T(p,p')}^{\overline{u}}(L_r(\lambda)) = \sum_{\mu\in P^+} m^\mu_{\lambda,p}\cdot \qdim(L_r(\mu)) \cdot \theta_{L_r(\mu)}^{p'/p}.
\label{eqn:torus_fd_unnorm}
\end{align}
If we now alter the framing to have writhe $0$, and then normalize, we have:
\begin{align*}
J_{T(p,p')}^{u}(L_r(\lambda)) = \theta_{L_r(\lambda)}^{-pp'}\sum_{\mu\in P^+} m^\mu_{\lambda,p}\cdot \qdim(L_r(\mu))\cdot\theta_{L_r(\mu)}^{p'/p},\\
J_{T(p,p')}^{u,\mathbbm{1}}(L_r(\lambda)) = \dfrac{\theta_{L_r(\lambda)}^{-pp'}}{\qdim(L_r(\lambda))}\sum_{\mu\in P^+} m^\mu_{\lambda,p}\cdot \qdim(L_r(\mu))\cdot\theta_{L_r(\mu)}^{p'/p}.
\end{align*}

\section{\texorpdfstring{$L_r(n\Lambda_1)$}{Lr(nLambda1)}-coloured Jones polynomials}
\label{sec:poly}
We now calculate $J^{\overline{u}}_{T(p,p')}(L_r(n\Lambda_1))$.
Recall that $u$ denotes the zero framing of the unknot and $\overline{u}$ denotes the inherited framing
on the $(p,p')$ cabling $T(p,p')$.
For simplicity, we shall omit the $\overline{u}$.

We begin by understanding and extending various components of the formula \eqref{eqn:torus_fd_unnorm}.

\subsection{Extending the twist to \texorpdfstring{$P$}{P}}
First is the twist. If we view $\frac{1}{2}(\lambda,\lambda+2\weylvec)$ 
as a polynomial function of $\lambda\in\la{h}^*$, it is not invariant with respect to the Weyl group.
However, for $\lambda\in P_r^+$, consider $\lambda'=\lambda+\weylvec$. 
Then, $\lambda'\in P_r^\circ$  and we have:
$\frac{1}{2}(\lambda,\lambda+2\weylvec)=\frac{1}{2}(\lambda'-\weylvec,\lambda'+\weylvec)=\frac{1}{2}(\lambda',\lambda')-\frac{1}{2}(\weylvec,\weylvec)$.
Now, extended to all of $\lambda'\in\la{h}^*$, this is clearly $\fS_r$ invariant.
We thus define:
\begin{align}
\Theta_{\lambda'} = q^{\frac{1}{2}(\lambda',\lambda')-\frac{1}{2}(\weylvec,\weylvec)}
\label{eqn:ThetaFormula}
\end{align}
and we have:
\begin{align}
\Theta_{w\lambda'}&=\Theta_{\lambda'},\label{eqn:Thetaweyl}\\
\theta_{L_r(\lambda)} &= \Theta_{\lambda+\weylvec}\label{eqn:thetaTheta}
\end{align}
for all $\lambda'\in\la{h}^*$ (we really only require $\lambda'\in P$) and all $\lambda\in P^+_r$.

\subsection{Extending the plethysm multiplicities to \texorpdfstring{$P_r$}{Pr}} 
Let $\lambda\in P_r^+$  and 
let 
\begin{align}
\psi_p(\ch(L_r(\lambda)))&=\sum_{\mu\in P^+} m^\mu_{\lambda,p}
\ch(L_r(\mu))
=\dfrac{1}{\Delta_r}\sum_{\mu\in P^+_r}\sum_{w\in \fS_r}(-1)^{\ell(w)} m^\mu_{\lambda,p}\mathbf{x}^{w(\mu+\weylvec)}.
\end{align}
Given any $\mu'\in P^\circ_r$ (so that $\mu=\mu'-\weylvec\in P^+_r$), we define a new function
\begin{align}
M^{\mu'}_{\lambda,p}=m^{\mu'-\weylvec}_{\lambda,p},
\label{eqn:Mshift}
\end{align}
and extend this to an alternating function on the union of Weyl translates of $P^\circ_r$ by
\begin{align}
M^{\mu''}_{\lambda,p}=(-1)^{\ell(w)}m^{\mu'-\weylvec}_{\lambda,p},
\label{eqn:MWeyl}
\end{align}
where $\mu''\in \fS_rP^\circ_r$, $w\in \fS_r$, $\mu'\in P^\circ_r$ such that $\mu''=w\mu'$.
We further define 
\begin{align}
M^{\mu''}_{\lambda,p}=0
\label{eqn:Mboundary}
\end{align} 
whenever $\mu''\in P_r \backslash \fS P^\circ_r$, which is the unique way to extend this 
function on all of $P_r$ so that it remains alternating.
We thus have:
\begin{align}
\psi_p(\ch(L_r(\lambda)))
&=\dfrac{1}{\Delta_r}\sum_{\mu''\in P_r} M^{\mu''}_{\lambda,p}\mathbf{x}^{\mu''}.
\label{eqn:extendmu}
\end{align}

Now we specialize to $\lambda=n\Lambda_1$. 
The crucial fact driving all of our calculations is that each weight of $L_r(n\Lambda_1)$ appears with multiplicity $1$. 
On the one hand we have:
\begin{align}
\psi_p(\ch(L_r(n\Lambda_1)))=\sum_{\lambda\in \Pi_{r,n}} \mathbf{x}^{p\lambda}.
\end{align}
On the other hand, we have \eqref{eqn:extendmu}.
Combining these two, the  numbers
$M^{\mu}_{n\Lambda_1,p}$, and in turn the numbers $m^\mu_{n\Lambda_1,p}$ are determined by the equation:
\begin{align*}
\Delta_r\cdot \psi_p(\ch(L_r(n\Lambda_1)))
=\Delta_r\sum_{\lambda\in \Pi_{r,n}} \mathbf{x}^{p\lambda} = \sum_{\mu\in P} M^{\mu}_{n\Lambda_1,p}\mathbf{x}^\mu.
\end{align*}
We may also write 
\begin{align}
\Delta_r\sum_{\lambda\in \Pi_{r,n}} \mathbf{x}^{p\lambda}
=\sum_{\lambda\in \Pi_{r,n},w\in \fS_r} (-1)^{\ell(w)}\mathbf{x}^{p\lambda+w\weylvec},
\end{align}
using the explicit expression for the Weyl denominator $\Delta_r$.
This finally gives us:
\begin{align}
\sum_{\mu\in P_r} M^{\mu}_{n\Lambda_1,p}\mathbf{x}^\mu
=\sum_{\lambda\in \Pi_{r,n},w\in \fS_r} (-1)^{\ell(w)}\mathbf{x}^{p\lambda+w\weylvec}.
\label{eqn:Mformula}
\end{align}

\subsection{Jones polynomials}
We now combine various extensions above with the Rosso--Jones formula to 
finally deduce our formula for the $L_r(n\Lambda_1)$ coloured invariants
of the torus knots. The derivation is a straight-forward generalization from the 
$\la{sl}_2$ case previously studied by Morton \cite{Mor-coloured}.
\begin{thm}
Let $p,p'$ be a pair of positive coprime integers.
We have the following formula for the framing dependent (with framing inherited from $p,p'$ cabling of $\circ$) un-normalized invariant.
\begin{align}
J_{T(p,p')}(L_r(n\Lambda_1))
&=
\dfrac{q^{-\frac{p'}{2p}\|\weylvec\|^2}}{\qdim(\Delta_r)}
\sum_{\lambda\in \Pi_{r,n},w\in\fS_r}(-1)^{\ell(w)}
q^{\frac{p'}{2p}\| p\lambda+w\weylvec\|^2+(p\lambda+w\weylvec,\weylvec)}.
\label{eqn:mainjones}
\end{align}

\end{thm}

\begin{proof}
We have:
\begin{align*}
J_{T(p,p')}(L_r(n\Lambda_1))&=
\sum_{\mu\in P^+_r}m^\mu_{n\Lambda_1,p}\cdot \qdim(L_r(\mu))\cdot \theta_{L_r(\mu)}^{p'/p}\\
&=\qdim\left(\sum_{\mu'\in P^\circ_r}m^{\mu'-\weylvec}_{n\Lambda_1,p}\cdot \ch(L_r(\mu'-\weylvec))\cdot \theta_{L_r(\mu'-\weylvec)}^{p'/p}\right)\\
&=\qdim\left(\frac{1}{\Delta_r}\sum_{\mu'\in P^\circ_r}\sum_{w\in\fS_r}(-1)^{\ell(w)}m^{\mu'-\weylvec}_{n\Lambda_1,p}\cdot \mathbf{x}^{w\mu'}\cdot \theta_{L_r(\mu'-\weylvec)}^{p'/p}\right)\\
&\overset{\eqref{eqn:thetaTheta}, \eqref{eqn:MWeyl}}{=}\qdim\left(\frac{1}{\Delta_r}\sum_{\mu\in \fS_rP^\circ_r}M^{\mu}_{n\Lambda_1,p}\cdot \mathbf{x}^{w\mu}\cdot \Theta_{\mu}^{p'/p}\right)\\
&\overset{\eqref{eqn:Mboundary}}{=}\qdim\left(\frac{1}{\Delta_r}\sum_{\mu\in P_r}M^{\mu}_{n\Lambda_1,p}\cdot \mathbf{x}^{w\mu}\cdot \Theta_{\mu}^{p'/p}\right)\\
&\overset{\eqref{eqn:Mformula}}{=}\qdim\left(\frac{1}{\Delta_r}\sum_{\lambda\in\Pi_{r,n},w\in \fS_r}(-1)^{\ell(w)}\mathbf{x}^{p\lambda+w\weylvec}\cdot \Theta_{p\lambda+w\weylvec}^{p'/p}\right)\\
&\overset{\eqref{eqn:qdimx}, \eqref{eqn:ThetaFormula}}{=}\frac{q^{-\frac{p'}{2p}\|\weylvec\|^2}}{\qdim(\Delta_r)}
\sum_{\lambda\in\Pi_{r,n},w\in \fS_r}(-1)^{\ell(w)}q^{(p\lambda+w\weylvec,\weylvec)+\frac{p'}{2p}\|p\lambda+w\weylvec\|^2}.
\end{align*}
\end{proof}

\section{Limits}
\label{sec:limits}

In this section, we calculate the limits of the Jones polynomials that we have found.
We expect these to equal the tails when $r\leq p$.
Recall that $J_K$ without any superscripts is the framing dependent and non-normalized invariant. 

\subsection{The case \texorpdfstring{$r\leq p$}{r<=p}}

There exists a unique element $w_0\in \fS_r$ such that $w_0(\Phi_r^+)=\Phi_r^-$, and  
in particular $w_0\weylvec=-\weylvec$. With respect to the simple transpositions $(i,i+1)$
this is the longest element of $\fS_r$ and its length is $\ell(w_0)=\vert \Phi_r^+\vert$.
\begin{thm}
Let $p,p'$ be coprime positive integers. Let $j$ be such that $0\leq j\leq r-1$.
Then, we have:
\begin{align}
\lim\limits_{n\rightarrow\infty} &J_{T(p,p')}(L_r((j+nr)\Lambda_1))
=
\dfrac{(q;q)^{r-1}_\infty}
{\prod_{\alpha\in\Phi_r^+}
(1-q^{(\alpha,\weylvec)})}
\cdot
\overline{\chi}^{r,p,p';\,j\Lambda_1}_{0,0}
\label{eqn:lim}
\end{align}
\end{thm}
\begin{proof}
From \eqref{eqn:mainjones} we have:
\begin{align}
J_{T(p,p')}(L_r(n\Lambda_1))
&=\frac{1}{\qdim(\Delta_r)}
\sum_{\lambda\in\Pi_{r,n},w\in \fS_r}(-1)^{\ell(w)}q^{\frac{pp'}{2}\|\lambda\|^2+p'(\lambda,w\weylvec)+p(\lambda,\weylvec)+(w\weylvec,\weylvec)}
\nonumber\\
&=
\frac{(-1)^{\ell(w_0)}}{\qdim(\Delta_r)}
\sum_{\lambda\in \Pi_{r,n},w\in \fS_r}(-1)^{\ell(ww_0)}q^{\frac{pp'}{2}\|\lambda\|^2-p'(\lambda,ww_0\weylvec)+p(\lambda,\weylvec)-(ww_0\weylvec,\weylvec)}
\nonumber\\
&=
\frac{(-1)^{\ell(w_0)}}{\qdim(\Delta_r)}
\sum_{\lambda\in \Pi_{r,n},w\in \fS_r}(-1)^{\ell(w)}q^{\frac{pp'}{2}\|\lambda\|^2-p'(\lambda,w\weylvec)+p(\lambda,\weylvec)-(w\weylvec,\weylvec)}
\nonumber\\
&=
\frac{(-1)^{\ell(w_0)}}{\qdim(\Delta_r)}
\sum_{w\in \fS_r}
\sum_{\lambda\in w^{-1}(\Pi_{r,n})}(-1)^{\ell(w)}q^{\frac{pp'}{2}\|w\lambda\|^2-p'(\lambda,\weylvec)+p(w\lambda,\weylvec)-(w\weylvec,\weylvec)}
\nonumber\\
&=
\frac{(-1)^{\ell(w_0)}}{\qdim(\Delta_r)}
\sum_{w\in \fS_r}
\sum_{\lambda\in w^{-1}(\Pi_{r,n})}(-1)^{\ell(w^{-1})}q^{\frac{pp'}{2}\|\lambda\|^2-p'(\lambda,\weylvec)+p(\lambda,w^{-1}\weylvec)-(\weylvec,w^{-1}\weylvec)}
\nonumber\\
&=
\frac{(-1)^{\ell(w_0)}}{\qdim(\Delta_r)}
\sum_{w\in \fS_r,\lambda\in \Pi_{r,n}}(-1)^{\ell(w^{-1})}q^{\frac{pp'}{2}\|\lambda\|^2-p'(\lambda,\weylvec)+p(\lambda,w^{-1}\weylvec)-(\weylvec,w^{-1}\weylvec)}
\nonumber\\
&=
\frac{(-1)^{\ell(w_0)}q^{-\|\weylvec\|^2}}{\qdim(\Delta_r)}
\sum_{\lambda\in \Pi_{r,n},w\in \fS_r}(-1)^{\ell(w)}q^{\frac{pp'}{2}\|\lambda\|^2-p'(\lambda,\weylvec)+p(\lambda,w\weylvec)-(\weylvec,w\weylvec)+(\weylvec,\weylvec)}
\label{eqn:altjones}
\end{align}
Note that we further have:
\begin{align}
\dfrac{(-1)^{\ell(w_0)}q^{-\|\weylvec\|^2}}{\qdim(\Delta_r)}
=
\dfrac{(-1)^{\vert\Phi_r^+\vert}q^{-\frac{1}{2}\sum_{\alpha\in\Phi_r^+}(\alpha,\weylvec)}}
{\prod_{\alpha\in\Phi_r^+}
(q^{\frac{1}{2}(\alpha,\weylvec)}-q^{-\frac{1}{2}(\alpha,\weylvec)} )}
=
\dfrac{1}
{\prod_{\alpha\in\Phi_r^+}
(1-q^{(\alpha,\weylvec)})}.
\label{eqn:prefactor}
\end{align}
The required limit now follows by recalling that $\Pi_{r,i}\subseteq \Pi_{r,i+r}$ (for all $i\geq 0$),
$\bigcup_{n\geq 0}\Pi_{r,j+nr}=j\Lambda_1+Q_r$
and finally comparing with \eqref{eqn:Wnormchar}.
\end{proof}

\begin{rem}
\label{rem:jonesintegral}
Let $n=kr$ in \eqref{eqn:altjones} for some positive integer $k$. In this case, $\Pi_{r,rk}\subset Q_r$ 
since $r\Lambda_1\in Q_r$. So, $\frac{1}{2}\|\lambda\|^2\in \ZZ$ by evenness of $Q_r$. 
Additionally, $(\lambda,\weylvec), (\lambda,w\weylvec), (\weylvec,\weylvec-w\weylvec)\in \ZZ$.
Combined with \eqref{eqn:prefactor}, we now see that $J_{T(p,p')}(L_r(kr\Lambda_1))\in \ZZ[[q]]$.
\end{rem}


For a knot $K$, let $\widehat{J}_K^{\bullet}(L_r(n\Lambda_1))$ denote $J_K^{\bullet}(L_r(n\Lambda_1))$ 
divided by its trailing monomial
(where $\bullet$ stands for either framed or unframed, normalized or un-normalized invariant).
Note that framing has no effect on $\widehat{J}^\bullet$.

\begin{defi}
\label{defi:tail}
Fix $r\geq 2$, $a\geq 1$ and $0\leq b\leq a-1$.
We say that the sequence $\widehat{J}_K^{\bullet}(L_r((an+b)\Lambda_1))$ 
is $0$-stable if there exists
a power series  $f\in \ZZ[[q]]$ such that for all $n$,
\begin{align}
\widehat{J}_K^{\bullet}(L_r((an+b)\Lambda_1)) - f(q) \in q^{(an+b)+1}\ZZ[[q]].
\end{align}
In this case, we say that $f$ is the $0$-tail of the sequence $\widehat{J}_K^{\bullet}(L_r((an+b)\Lambda_1))$.
\end{defi}

Several comments are now in order, encapsulated in the following remark.

\begin{rem}
\label{rem:limits}
Fix $r\leq p<p'$ with $p,p'$ coprime.
\begin{enumerate}
\item When $j=0$, the limit \eqref{eqn:lim} involves $\overline{\chi}^{r,p,p'}_{0,0}$ which 
is modular \cite{DonLinNg}, \cite{Zhu} up to some factor $q^t$, since it
is the character of $\sW_r(p,p')$ which is in turn a
rational \cite{Ara-princrat} and $C_2$ cofinite VOA \cite{Ara-c2}.
Moreover, using \eqref{eqn:chi1plus} we see that the RHS of \eqref{eqn:lim} $\in 1+q\ZZ[[q]]$.
This means that for all but finitely many $n$, $J_{T(p,p')}(L_r(rn\Lambda_1))=\widehat{J}_{T(p,p')}(L_r(rn\Lambda_1))\in  1+q\ZZ[[q]]$.

\item \label{item:0tail} Let us continue to take $j=0$.
It is clear that:
\begin{align*}
\left(\lim_{n\rightarrow\infty}J_{T(p,p')}(L_r(nr\Lambda_1))\right) 
&- J_{T(p,p')}(L_r(nr\Lambda_1)) \\
&= \dfrac{q^{-\frac{p'}{2p}\|\weylvec\|^2}}{\qdim(\Delta_r)}\sum_{\lambda\in Q_r\backslash \Pi_{r,nr}}
\sum_
{w\in \fS_r}(-1)^{\ell(w)}q^{\frac{p'}{2p}\|p\lambda+w\weylvec\|^2+(p\lambda+w\weylvec,\weylvec)}.
\end{align*}
Now, since $\|\cdot\|^2$ is a positive definite quadratic form, the minimum $q$ degree of the right-hand side grows 
as a quadratic in $n$. This means that for all large enough $n$, $J_{T(p,p')}(L_r(nr\Lambda_1))$ matches the limit
\eqref{eqn:lim}
at least up to first $q^{g(n)}$ terms for some quadratic function $g$. In effect, for all large enough $n$,
we must in particular have a match for the first $nr$ terms. However, at the moment, we do not know how to establish
that this match works for all $n$ which would prove that the RHS of \eqref{eqn:lim} is indeed the $0$-tail.

\item For $j\neq 0$, the modular properties of $\overline{\chi}^{r,p,p';\,j\Lambda_1}_{0,0}$ are as yet unclear to us.
Moreover, the minimum degree of $\overline{\chi}^{r,p,p';\,j\Lambda_1}_{0,0}$ seems to depend on $r,p,p',j$.

\item In the special case $r=p$, \eqref{eqn:r=p} gives that the limit is essentially independent of $j$, up to a sign. This means
that in this case we have:
\begin{align}
\lim_{n\rightarrow\infty}\widehat{J}_{T(p,p')}(L_r(n\Lambda_1))=
\dfrac{(q;q)^{r-1}_\infty}
{\prod_{\alpha\in\Phi_r^+}
(1-q^{(\alpha,\weylvec)})}
\overline{\chi}^{r,r,p'}_{0,0}. 
\label{eqn:r=plim}
\end{align}
\item Continuing to take $r=p$, similar to point \ref{item:0tail} above, we expect the RHS of \eqref{eqn:r=plim}
to be the $0$-tail of the sequence 
$J_{T(p,p')}(L_r(n\Lambda_1))$. In this case, $n$ is not confined to multiples of $r$.
\item In the case of $\la{sl}_2$, there is another way in which $r=p=2$ case is somewhat special -- the coloured Jones polynomials
for $T(p,p')$ satisfy a second order difference equation which reduces to first order for the torus knots $T(2,2m+1)$ \cite{Hik-diff}.
\end{enumerate}
\end{rem}

\subsection{A case study for \texorpdfstring{$p<r$}{p<r}}
This case is quite a bit harder, since the limit \eqref{eqn:lim} in this case is $0$ due to Corollary \ref{prop:p<r}. 
This means that there is cancellation in the formula for the Jones polynomial
and it is quite tricky to pin point their minimum degrees. 
For general $p<r$, this appears to be a fairly involved problem, which we hope to address in the future.
For now, we analyze the case $p=2,r=3$ to ensure that our results match with those given in \cite{GarVuo}.
We expect several considerations of this example to generalize to higher $r$.
The following theorem gives the formula for $J_{T(2,p')}(L_3(n\Lambda_1))$ which agrees 
up to shifting, sign, framing and normalization with the one from \cite{GarVuo}.
For more calculations regarding $\la{sl}_3$ coloured invariants of $T(2,n)$, see \cite{GarMorVuo}.
\begin{thm}
\label{thm:p=2}
For any odd number $p'\geq 3$ we have:
\begin{align}
J_{T(2,p')}(L_3(n\Lambda_1))=
\dfrac{(-1)^{n}q^{\frac{p'}{2}(\frac{n^2}{3}+n)-n}}{(1-q)^2(1-q^2)}
\sum_{i=0}^{n} 
(-1)^{i}
q^{\frac{p'}{2}(i^2+i)-i}
(1-q^{n-i+1})(1-q^{2i+1})(1-q^{n+i+2})
\label{eqn:p=2}
\end{align}
and so,
\begin{align}
\lim_{n\rightarrow\infty}\widehat{J}_{T(2,p')}(L_3(n\Lambda_1))=
\dfrac{1}{(1-q)^2(1-q^2)}
\sum_{i=0}^{\infty} 
(-1)^{i}
q^{\frac{p'}{2}(i^2+i)-i}
(1-q^{2i+1})
=\dfrac{(q;q)_\infty\cdot \overline{\chi}^{2,2,p'}_{0,0}}{(1-q)^2(1-q^2)}
\label{eqn:p=2lim}
\end{align}
\end{thm}
\begin{proof}
Various features and steps of this proof are depicted in Figure \ref{fig:p=2} for the case $n=4$.

In the case of $\la{sl}_3$ we have 
$\weylvec = \epsilon_1-\epsilon_3 = \theta$ where $\theta\in\Phi_3$ is the highest root and moreover, in this case
$\fS_3\weylvec=\Phi_3$.
From Lemma \ref{lem:cancellations}, we have $w\weylvec-\weylvec\in 2Q_3$ with $1\neq w\in\fS_3$ iff $w=(1,3)$ and in this case,
$(1,3)\weylvec-\weylvec=-2\theta$.

Now let $n\geq 1$.
Equation \eqref{eqn:mainjones} gives:
\begin{align}
J_{T(2,p')}(L_3(n\Lambda_1))
=
\dfrac{q^{-\frac{p'}{4}\|\weylvec\|^2}}{\qdim(\Delta_3)}\sum_{(\lambda,w)\in \Pi_{3,n}\times \fS_3}
(-1)^{\ell(w)}
q^{\frac{p'}{4}\| 2\lambda+w\weylvec\|^2+(2\lambda+w\weylvec,\weylvec)}.
\label{eqn:sl3_2}
\end{align}
This sum immediately involves cancellations which we now explain.
Given $\lambda\in \Pi_{3,n}$, suppose there exist $\widetilde{\lambda}\in \Pi_{3,n}$ and $\alpha\in\Phi_3$
such that $\widetilde{\lambda}=\lambda+\alpha$. 
There is now a unique $w\in\fS_3$ such that
$\alpha=w\theta$ (among $\la{sl}_r$, such a uniqueness only holds for $\la{sl}_2$ and $\la{sl}_3$.)
Let $\widetilde{w}=w(1,3)$.
We now have
$2\lambda+w\weylvec=2\widetilde{\lambda}+\widetilde{w}\weylvec$ and $(-1)^{\ell(w)}\neq (-1)^{\ell(\widetilde{w})}$.
This crucially implies that the contributions corresponding to $(\lambda,w)$ and $(\widetilde{\lambda},\widetilde{w})$ cancel
in \eqref{eqn:sl3_2}. Consequently, we may omit these terms from the summation.
The only terms $(\lambda,w)$ that now contribute to the summation are such that $\lambda+w\theta\not\in \Pi_{3,n}$.

It is now easy to see that $\lambda$ must belong to the boundary of $\Pi_{3,n}$. 
There are three sides to this boundary:
\begin{align*}
\lambda\in 
\{n\Lambda_1 - i\alpha_1\,|\, 0\leq i\leq n\} \cup
\{n\Lambda_1 - n\alpha_1-i\alpha_2\,|\, 0\leq i\leq n\} \cup
\{n\Lambda_1 - i\alpha_1-i\alpha_2\,|\, 0\leq i\leq n\}.
\end{align*}
In the first case, we have $w\theta= (1,2)\theta=\alpha_2 $ or $w\theta=\theta=\alpha_1+\alpha_2$;
in the second, we have $w\theta= (3,1,2)\theta=-\alpha_1 $ or $w\theta=(1,3)\theta=-\alpha_1-\alpha_2$;
in the third, we have $w\theta= (2,3)\theta=\alpha_1 $ or $w\theta=(1,3,2)\theta=-\alpha_2$.

Writing $2\lambda+w\weylvec = \mu$, we now see that each term in \eqref{eqn:sl3_2}
is of the form 
$$(-1)^{\ell(w)}  q^{(\mu, \frac{p'}{4}\mu+\weylvec)},$$
with $\mu$ belonging to one of the following six sets, along with the appropriate sign $(-1)^{\ell(w)}$:
\begin{align*}
S_1 = \{2n\Lambda_1 - 2i\alpha_1 + \alpha_1+\alpha_2\,|\, 0\leq i\leq n\}, &\quad (-1)^{\ell(w)}=1,\\
S_2 = \{2n\Lambda_1 - 2i\alpha_1 + \alpha_2\,|\, 0\leq i\leq n\}, &\quad (-1)^{\ell(w)}=-1,\\
S_3 = \{2n\Lambda_1 - 2n\alpha_1-2i\alpha_2-\alpha_1\,|\, 0\leq i\leq n\},  &\quad (-1)^{\ell(w)}=1,\\
S_4 = \{2n\Lambda_1 - 2n\alpha_1-2i\alpha_2-\alpha_1-\alpha_2\,|\, 0\leq i\leq n\},  &\quad (-1)^{\ell(w)}=-1,\\
S_5 = \{2n\Lambda_1 - 2i\alpha_1-2i\alpha_2-\alpha_2\,|\, 0\leq i\leq n\}, &\quad (-1)^{\ell(w)}=1,\\
S_6 = \{2n\Lambda_1 - 2i\alpha_1-2i\alpha_2+\alpha_1\,|\, 0\leq i\leq n\}, &\quad (-1)^{\ell(w)}=-1.
\end{align*}
That is, we have:
\begin{align*}
J_{T(2,p')}(L_r(n\Lambda_1))
=
\dfrac{q^{-\frac{p'}{4}\|\weylvec\|^2}}{\qdim(\Delta_3)}
\left(\sum_{\mu\in S_1} - \sum_{\mu\in S_2} +\sum_{\mu\in S_3}-\sum_{\mu\in S_4}+\sum_{\mu\in S_5}-\sum_{\mu\in S_6}\right)
q^{(\mu, \frac{p'}{4}\mu+\weylvec)}
\end{align*}
It is now not too hard to see that this sum can be rearranged by grouping together the Weyl translates of those 
$\mu\in S_1\cup\cdots \cup S_6$ that belong to the fundamental Weyl chamber.  
These are precisely the $\mu$ in $S_1$ with $0\leq i\leq \lfloor \frac{n}{2}\rfloor$
or $\mu$ in $S_2$ with $0\leq i\leq \lfloor \frac{n-1}{2}\rfloor$.
In other words, we have:
\begin{align*}
J_{T(2,p')}&(L_r(n\Lambda_1))
=
\dfrac{q^{-\frac{p'}{4}\|\weylvec\|^2}}{\qdim(\Delta_3)}
\left( \sum_{i=0}^{\lfloor\frac{n}{2}\rfloor} 
\sum_{w\in\fS_3}
(-1)^{\ell(w)}
q^{(w( 2n\Lambda_1-2i\alpha_1+\alpha_1+\alpha_2), \frac{p'}{4}w(2n\Lambda_1-2i\alpha_1+\alpha_1+\alpha_2) +\weylvec)}
\right.\nonumber\\
-&\left.
\sum_{i=0}^{\lfloor\frac{n-1}{2}\rfloor} 
\sum_{w\in\fS_3}
(-1)^{\ell(w)}
q^{(w( 2n\Lambda_1-2i\alpha_1+\alpha_2), \frac{p'}{4}w(2n\Lambda_1-2i\alpha_1+\alpha_2) +\weylvec)}
\right)\\
&=
\dfrac{q^{-\frac{p'}{4}\|\weylvec\|^2}}{\qdim(\Delta_3)}
\sum_{i=0}^{n} 
(-1)^{i}
\sum_{w\in\fS_3}
(-1)^{\ell(w)}
q^{(w( 2n\Lambda_1-i\alpha_1+\alpha_1+\alpha_2), \frac{p'}{4}w(2n\Lambda_1-i\alpha_1+\alpha_1+\alpha_2) + \weylvec)}
	\nonumber\\
&=
\dfrac{q^{-\frac{p'}{4}\|\weylvec\|^2}}{\qdim(\Delta_3)}
\sum_{i=0}^{n} 
(-1)^{i}
q^{\frac{p'}{4}(\frac{8n^2}{3}+2i^2-2i+2-4ni+4n)}
\sum_{w\in\fS_3}
(-1)^{\ell(w)}
q^{(w( 2n\Lambda_1-i\alpha_1+\alpha_1+\alpha_2),  \weylvec)}
	\nonumber\\
&=\dfrac{q^{-\frac{p'}{4}\|\weylvec\|^2}}{\qdim(\Delta_3)}
\sum_{i=0}^{n} 
(-1)^{i+1}
q^{\frac{p'}{4}(\frac{8n^2}{3}+2i^2-2i+2-4ni+4n)}
q^{-2-2n+i}
(1-q^{i+1})(1-q^{2n-2i+1})(1-q^{2n-i+2})
	\nonumber\\
&=\dfrac{(-1)^{n+1}q^{\frac{p'}{2}(\frac{n^2}{3}+n)-2-n}}{\qdim(\Delta_3)}
\sum_{i=0}^{n} 
(-1)^{i}
q^{\frac{p'}{2}(i^2+i)-i}
(1-q^{n-i+1})(1-q^{2i+1})(1-q^{n+i+2}),
\end{align*}
as required. Now the limit of shifted Jones polynomials follows easily.
\end{proof}

Just like the $r=p$ case above, we see that the limit in this
example is also independent of the congruence class of $n$ modulo $r$.
Computer experiments suggest that this continues to happen for higher values
of $r$ and with $p<r$. In fact, experiments with $r=3,4,5$ suggest the following precise conjecture.
Note that the case $p=r$ proved above also fits the pattern of this conjecture.
\begin{conj} \label{conj:p<r}
Let $2\leq p<p'$ be a pair of coprime positive integers and let $2\leq r$ be such that $p<r$.
Letting $\delta_r$ denote the Weyl vector of $\la{sl}_r$ and $\Phi^+_0=\Phi^+_1=\emptyset$ we have:
\begin{align}
\lim_{n\rightarrow\infty}\widehat{J}_{T(p,p')}(L_r(n\Lambda_1))
=\dfrac{\prod_{\alpha\in \Phi^{+}_{r-p}} (1-q^{(\alpha,\weylvec_{r-p})})}{\prod_{\alpha\in \Phi^{+}_{r}} (1-q^{(\alpha,\weylvec_{r})})}
(q;q)^{p-1}_\infty \cdot\overline{\chi}^{p,p,p'}_{0,0}.
\end{align}
\end{conj}

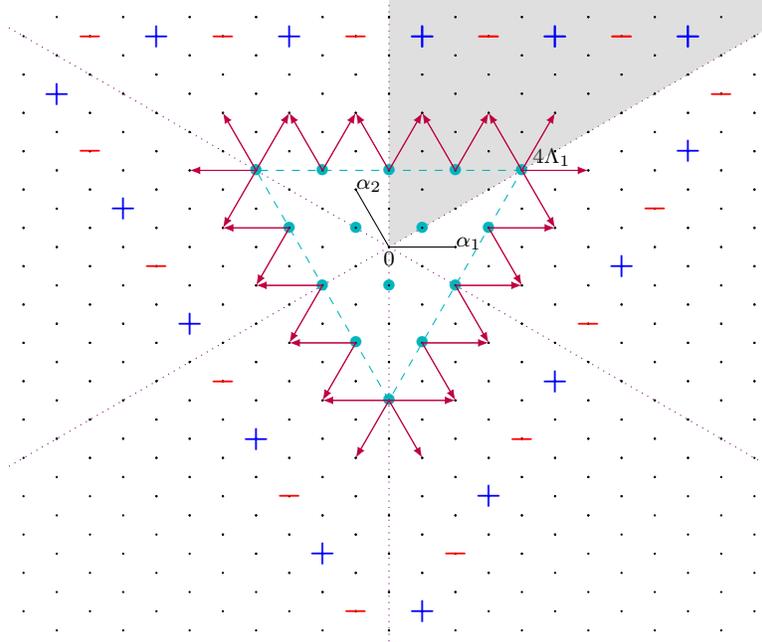
\begin{figure}
\begin{tikzpicture}[x={(30:0.6cm)},y={(150:0.6cm)},scale=0.85]
\clip (-6cm,-6.2cm) rectangle (6cm,3.9cm);

	\draw[thin] (0,0) -- (1,-1);
	\draw[thin] (0,0) -- (1,2);

	\node (or) at (0,0) {};
	\node at (-0.3,-0.3) {\scriptsize{$0$}};

	\node (a1) at (1,-1) {};
	\node at (1.3,-1.1) {\scriptsize{$\alpha_1$}};
	\node (a2) at (1,2) {};
	\node at (1.3,1.9) {\scriptsize{$\alpha_2$}};
	\node (l1) at (1,0) {};
	\node (l2) at (1,1) {};
	\node at (4.8,-0.1) {\scriptsize{$4\Lambda_1$}};

	\draw[violet,dotted] (0,0) -- ($20*(l1)$);
	\draw[violet,dotted] (0,0) -- ($20*(l2)$);
  	
  	\draw[violet,dotted] (0,0) -- ($-20*(l1)$);
	\draw[violet,dotted] (0,0) -- ($-20*(l2)$);

	\draw[violet,dotted] (0,0) -- ($20*(l1)-20*(a1)$);
	\draw[violet,dotted] (0,0) -- ($-20*(l1)+20*(a1)$);
  	
	\fill [opacity=0.25,gray] (0,0) -- ($20*(l1)$) -- ($20*(l2)$) -- (0,0);

	\foreach \x in {-20,...,20}
    	\foreach \y in {-20,...,20}
    	{
    		\fill (\x,\y) circle (0.5pt);
    	}

	\newcommand{\moduleweights}[3]{
		\foreach \x in {0,...,#1}{
			\foreach \y in {0,...,\x}
			{
				\node at ($#1*(l1)-\x*(a1)-\y*(a2)$) {\textcolor{#2}{#3}};
			}
		}
		\draw[dashed,draw=#2,thin] ($ #1*(l1) $) -- ($#1*(l1)-#1*(a1)$)--($#1*(l1)-#1*(a1)-#1*(a2)$) -- ($ #1*(l1) $);
	}

	\moduleweights{4}{Aquamarine}{{$\bullet$}};

 	\foreach \x in {-10,...,10}
    	\foreach \y in {-10,...,10}
    	{
    		\fill (\x,\y) circle (0.5pt);
    	}

	\foreach \x in {0,...,4}{
		\draw[-latex,draw=purple,fill=purple] ($4*(l1)-\x*(a1)$) -- ($4*(l1)-\x*(a1)+(a1)+(a2)$);
		\draw[-latex,draw=purple,fill=purple] ($4*(l1)-\x*(a1)$) -- ($4*(l1)-\x*(a1)+(a2)$);

		\draw[-latex,draw=purple,fill=purple] ($4*(l1)-4*(a1)-\x*(a2)$) -- ($4*(l1)-4*(a1)-\x*(a2)-(a1)$);
		\draw[-latex,draw=purple,fill=purple] ($4*(l1)-4*(a1)-\x*(a2)$) -- ($4*(l1)-4*(a1)-\x*(a2)-(a1)-(a2)$);

		\draw[-latex,draw=purple,fill=purple] ($4*(l1)-\x*(a1)-\x*(a2)$) -- ($4*(l1)-\x*(a1)-\x*(a2)+(a1)$);
		\draw[-latex,draw=purple,fill=purple] ($4*(l1)-\x*(a1)-\x*(a2)$) -- ($4*(l1)-\x*(a1)-\x*(a2)-(a2)$);
	}

	\foreach \x in {0,...,4}{
		\node at ($8*(l1)-2*\x*(a1)+(a1)+(a2)$) {\textcolor{blue}{{$\boldsymbol{+}$}}};
		\node at ($8*(l1)-2*\x*(a1)+(a2)$) {\textcolor{red}{{$\boldsymbol{-}$}}};

		\node at ($8*(l1)-8*(a1)-2*\x*(a2)-(a1)$) {\textcolor{blue}{{$\boldsymbol{+}$}}};
		\node at ($8*(l1)-8*(a1)-2*\x*(a2)-(a1)-(a2)$) {\textcolor{red}{{$\boldsymbol{-}$}}};

		\node at ($8*(l1)-2*\x*(a1)-2*\x*(a2)+(a1)$) {\textcolor{red}{{$\boldsymbol{-}$}}};
		\node at ($8*(l1)-2*\x*(a1)-2*\x*(a2)-(a2)$) {\textcolor{blue}{{$\boldsymbol{+}$}}};
	}

	\node at ($8*(l1)-0*(a1)+(a1)+(a2)$) {\textcolor{blue}{{$\boldsymbol{+}$}}};
	\node at ($8*(l1)-2*(a1)+(a1)+(a2)$) {\textcolor{blue}{{$\boldsymbol{+}$}}};
	\node at ($8*(l1)-4*(a1)+(a1)+(a2)$) {\textcolor{blue}{{$\boldsymbol{+}$}}};
	\node at ($8*(l1)-0*(a1)+(a2)$) {\textcolor{red}{{$\boldsymbol{-}$}}};
	\node at ($8*(l1)-2*(a1)+(a2)$) {\textcolor{red}{{$\boldsymbol{-}$}}};

\end{tikzpicture}
\caption{Proof of theorem \ref{thm:p=2} with $n=4$.
The points \textcolor{Aquamarine}{$\bullet$} denote weights of $L_3(4\Lambda_1)$. 
The arrows \textcolor{purple}{$\rightarrow$}
correspond to pairs $(\lambda,w)$ such that $\lambda\in \Pi_{3,4}$ but $\lambda+w\weylvec\not\in\Pi_{3,4}$.
These arrows originate at $\lambda$ and point in the direction of $w\weylvec$. 
Points \textcolor{red}{$\boldsymbol{-}$} and \textcolor{blue}{$\boldsymbol{+}$} denote the actual contributions to the 
sum, with blue points giving a positive contribution and red giving negative. 
The gray region denotes fundamental Weyl chamber. Points \textcolor{red}{$\boldsymbol{-}$} and \textcolor{blue}{$\boldsymbol{+}$} 
belonging to this region appear in the final summation \eqref{eqn:p=2}.
}
\label{fig:p=2}
\end{figure}


\providecommand{\oldpreprint}[2]{\textsf{arXiv:\mbox{#2}/#1}}\providecommand{\preprint}[2]{\textsf{arXiv:#1
  [\mbox{#2}]}}

\end{document}